\documentclass{amsart}
\usepackage{amsaddr}

\usepackage{amsmath,amssymb,amsthm} 
\usepackage{longtable}
\usepackage{epsf,stmaryrd,amscd} 
\usepackage{xspace}
\usepackage{graphicx,epstopdf}
\usepackage{latexsym,mathtools}
\usepackage{bbding} 
\usepackage[colorlinks=true,linkcolor=black,citecolor=black,urlcolor=black]{hyperref} 
\usepackage{multirow} 
\usepackage{array} 
\newcolumntype{P}[1]{>{\centering\arraybackslash}p{#1}} 
\usepackage{longtable} 
\usepackage{faktor}
\usepackage{listings} 
\usepackage[mark=o,radius=.3cm]{dynkin-diagrams}

\newcommand{\ra}{\rangle}
\newcommand{\la}{\langle}

\newcommand{\fg}{\mathfrak g} 
\newcommand{\fh}{\mathfrak h}

\newcommand{\fsl}{\mathfrak{sl}} 

\newcommand{\Tor}{{\mathrm{Tor}}}

\newcommand{\im}{{\mathrm{im}}} 

\newcommand{\Auh}{{\mathrm{Auh}\,}}

\renewcommand{\NG}{{\mathrm{N}_G}}
\newcommand{\CG}{{\mathrm{C}_G}}
\newcommand{\CL}{{\mathrm{Cl}_G}} 

\newcommand{\SL}{{\mathrm{SL}}}
\newcommand{\SU}{{\mathrm{SU}}} 
\newcommand{\PSL}{{\mathrm{PSL}}} 
 
\newcommand{\Sp}{{\mathrm{Sp}}} 
\newcommand{\SO}{{\mathrm{SO}}} 
\newcommand{\Spi}{{\mathrm{Spin}}}
\newcommand{\SSpi}{{\mathrm{SSpin}}} 
 
\newcommand{\bC}{{\mathbb C}}
\newcommand{\bd}{{\mathbf d}} 
\newcommand{\bF}{{\mathbb F}} 
\newcommand{\bH}{{\mathbb H}} 
 
\newcommand{\bK}{{\mathbb K}}

\newcommand{\bQ}{{\mathbb Q}}

\newcommand{\bZ}{{\mathbb Z}}

\newcommand{\mP}{{\mathcal P}}

\newtheorem{conh}{Conjecture}
\newtheorem{thmh}[conh]{Theorem}

\newtheorem{defn}{Definition}[section]
\newtheorem{theorem}[defn]{Theorem}

\newtheorem{lem}[defn]{Lemma}
\newtheorem{cor}[defn]{Corollary}

\newtheorem{prop}[defn]{Proposition}
\newtheorem{question}[defn]{Question}

\begin{document}
\title{Topological rigidity of $\SL_2$-quotients}
\author{Dylan Johnston}
\email{dpj42@hotmail.com}
\address{Department of Mathematics, University of Warwick, Coventry, CV4 7AL, UK}
\author{Dmitriy Rumynin (corresponding author)}
\email{D.Rumynin@warwick.ac.uk}
\address{Department of Mathematics, University of Warwick, Coventry, CV4 7AL, UK}
\date{September 13, 2023}

\begin{abstract}
We investigate  the homotopy type of a certain homogeneous space for a simple complex Lie group.
We calculate some of its classical topological invariants and introduce a new one.
We also propose several conjectures about its topological rigidity.  
\end{abstract}

\maketitle

Let $G$ be a simple simply-connected complex Lie group.
By Jacobson-Morozov Theorem, its unipotent conjugacy classes are in one-to-one correspondence with conjugacy classes
of homomorphisms:
\[
\varphi_u : \SL_2 (\bC)\rightarrow G
\qquad \longleftrightarrow \qquad 
u = \varphi_u \begin{pmatrix} 1&1\\0&1\end{pmatrix}\, .
\]
For each unipotent element $u\in G$ we consider the corresponding quotient topological space
$X_u = X_u(G) \coloneqq G/H_u$ where $H_u= \varphi_u (\SL_2 (\bC))$.
This space is an affine algebraic variety: it is well-known that $G/H$ is affine if and only if $H$ is reductive.
In the present paper we investigate the homotopy type of $X_u$.
The following is our first main result:
\begin{thmh} \label{main_theorem}
  Consider two simple simply-connected complex Lie groups $G$, $K$ and their
  unipotent elements $u\in G$, $v\in K$.
 If  $X_u(G)$ and $X_v(K)$ are homotopy equivalent, then $G\cong K$.
\end{thmh}

To prove this, we compute some of its classical topological invariants: rational homotopy type and
some of the homotopy groups.
Our second main result is related to this calculation:
$\pi_2 (X_u)$ depends on whether $u$ is quite even or not.
All quite even orbits are described in Section~\ref{VeryEven}.
Out second main result is
Theorem~\ref{QuiteEvenClassical}: it describes quite even orbits in classical types in terms of partitions.

Our third main result is Theorem~\ref{2nd_thm}: we show that a certain ideal $I_u\lhd \bZ[x]$ is a topological invariant
of $X_u$. This invariant originates  from the topological K-theory but is much easier to use.

In light of Theorem~\ref{main_theorem}, the homotopy classification of the spaces
$X_u(G)$ reduces to the spaces with the same group $G$.
We have collected enough evidence to propose the following conjecture.
\begin{conh} \label{wk_rigid} {\textrm (Weak topological rigidity)}
  Let $u,v\in G$ be unipotent elements, $H_u$ and $H_v$ -- their corresponding subgroups.
  The spaces  $X_u(G)$ and $X_v(G)$ are homotopy equivalent
  if and only if there exists an automorphism $\psi : G \rightarrow G$ such that $\psi (H_u) = H_v$.
\end{conh}

Such an isomorphism $\psi$ could be inner: then $u$ and $v$ belong to the same conjugacy class.
However, if the Dynkin diagram has diagram automorphisms, then
not all automorphisms of $G$ are inner and
different conjugacy classes can be related by $\psi$. See Section~\ref{Example2} for more detailed discussion.

Let us now put forward two further, quite bold conjectures. While there is little evidence to support them,  we feel that they are interesting question:
proving or disproving them would be remarkable.
Notice that we need to exclude $G=X_1(G)$ since the inverse element map
$x\mapsto x^{-1}$ is a homotopy equivalence, not homotopic to an equivariant map
  \begin{conh} \label{st_rigid} {\textrm (Strong topological rigidity)}
    Suppose $u,v\in G$, $v\neq 1\neq u$ and
    $\phi: X_u(G) \rightarrow X_v(G)$ is continuous map that is a homotopy equivalence.
    Then $\phi$ is homotopic to a $G$-equivariant map.  
\end{conh}
  All $G$-equivariant maps $X_u\rightarrow X_u$ are easy to describe:
  once the coset $H$ is mapped to a coset $zH$, the whole map has the form $xH \mapsto xzH$.
  To be well-defined, $z$ needs to belong to the normaliser $\NG(H)$.
Hence, we have a group homomorphism $\eta: \NG(H) \rightarrow \Auh (X_u)$, to the group of homotopy self-equivalences.
  
  Its component group $\pi_0 (\NG(H)) = \NG(H)/\NG(H)_0$ is the same as components groups of the centralisers
  $\pi_0 (\CG(H))$ and $\pi_0 (\CG(u))$. It also coincides with the fundamental group of the conjugacy class $\pi_1 (\CL(u))$.
Our final conjecture is about this group.
 \begin{conh}  The group homomorphism
$\pi_0 ( \eta ): \pi_0 (\NG(H)) \rightarrow \pi_0 (\Auh (X_u))$
is an isomorphism. 
\end{conh}

 Note that Conjecture~\ref{st_rigid} implies surjectivity of $\pi_0 ( \eta )$.
 
 Let us now provide a detailed
description of the content of the present paper.
We start by reminding the reader the basic facts about $X_u$ and its homotopy groups $\pi_n (X_u)$  in Section~\ref{intro}.

The next three sections are devoted to calculations of homotopy groups.
In Section~\ref{RaHoT}, we compute the rational homotopy type of $X_u$ and all the groups $\pi_n (X_u)\otimes \bQ$.
In Section~\ref{VeryEven} we describe all the quite even orbits,
essentially computing all $\pi_2 (X_u)$. 
We compute some of the higher homotopy groups 
$\pi_n (X_u)$ for $n \leq 6$ in Section~\ref{HiHoGr}: this completes our proof
of  Theorem~\ref{main_theorem}.

We turn our attention to K-theory in Section~\ref{KTheory}:
it can be computed using the Hodgkin spectral sequence and Kozsul resolution.
This yields $I_u$, a new homotopy invariant of $X_u$. 
In the penultimate Section~\ref{Example}, we employ our two techniques to tackle Conjecture~\ref{wk_rigid}
in dimension 75. This is an interesting example because these homogeneous spaces come from three different groups of dimension 78:
$B_6$, $C_6$ and $E_6$.
We are coming slightly short: cf. Proposition~\ref{finalprop}.
In the final Section~\ref{Example2}, we clarify the effect of the Dynkin diagram automorphisms.

The authors would like to thank
Michael Albanese,
John Greenlees,
John Jones and 
Miles Reid
for insightful information and/or inspirational conversations. 

  \section{Introduction} \label{intro}
  A thoughtful reader will notice that the same questions can be asked for compact groups. Indeed, these languages are equivalent (cf. \cite{JoRuTh}).
  More precisely, $G\simeq G_c$, i.e., $G$ is homotopy equivalent to its maximal compact subgroup, while $X_u \simeq G_c/H_c$ \cite[Lemma 2.5]{RuTa}.
  The latter is a compact oriented manifold, hence, its dimension is a homotopy invariant.
  \begin{prop} \label{Xu_dim}
    If $X_u \simeq X_v$, then $\dim (X_u)=\dim (X_v)$.
  \end{prop}
  
  For the rest of the section, assume that $u\neq 1$.
Let  $X=X_u(G)=G/H$, where $H$ is the image of $\varphi_u$. The quotient map
  $G \rightarrow X$ is a Serre fibration with path-connected $X$ and fibre $H$.
  This yields a long exact sequence in homotopy
  \begin{equation} \label{long_es}
    \dots \rightarrow \pi_i(H) \rightarrow \pi_i(G) \rightarrow \pi_i(X) \rightarrow \pi_{i-1}(H) \rightarrow \dots \rightarrow \pi_0(G) \rightarrow \pi_0(X) \rightarrow 0.
    \end{equation}
  Since we know some of the low homotopy groups
  \begin{equation} \label{homot_G}
    \pi_2(G)=\pi_2(H)=\pi_1(G)=0, \
  \pi_3(G)=\pi_3(H)=\bZ, 
  \end{equation}
  the end of the long exact sequence~\eqref{long_es} looks like
  \begin{equation} \label{long_es2}
  \dots \rightarrow \bZ \xrightarrow{\pi_3 (\varphi_u)} \bZ \rightarrow \pi_3(X) \rightarrow 0 \rightarrow 0 \rightarrow \pi_2(X)
  \rightarrow \pi_1(H) \rightarrow 0.
  \end{equation}
  The integer $\nu (u) \coloneqq \pi_3 (\varphi_u)$ is the Dynkin index of the Lie algebra homomorphism $d\varphi_u: \fsl_2 \rightarrow \fg$
  \cite[Theorem 17.2.2]{onishchik1994topology}.
  Their values are known for all $u$ \cite{Dynkin1952,PANYUSHEV201515}.

  Note that $H$ is either $\PSL_2 (\bC)$ and $\pi_1 (H) = \bZ/2$, or $\SL_2 (\bC)$ and $\pi_1 (H) = 0$.
  Let us call the unipotent $u$ {\em quite even} if $H=\PSL_2 (\bC)$.
This is reminiscent to the noting of even unipotent:
$u$ is even if and only if the image of  $\SL_2 (\bC)$ in the adjoint group $G_{ad}=G/Z(G)$ is $\PSL_2 (\bC)$.
Coupled with this information, 
  the long exact sequence~\eqref{long_es2} yields the three lower homotopy  groups
  \begin{prop} \label{low_pi}
  If $u\neq 1$ and $X_u = X_u(G)$, then 
    \[\pi_1(X_u)=0, \
  \pi_3(X_u)=\bZ/\nu (u), \
  \pi_2(X_u) = \begin{cases}
    \bZ/2, & \text{if } u \text{ is quite even}, \\
    0,& \text{otherwise}.
  \end{cases}\]
\end{prop}
  Note that if $u=1$, then the homotopy groups of $X_1=G$ are given in \eqref{homot_G}.
  This is somewhat consistent with Proposition~\ref{low_pi}: $\nu (1)=0$ because $d\varphi_1 =0$ and its Dynkin index is $0$.
  It remains to agree that $1$ is {\em not quite even} but it is inconsistent with the results of Section~\ref{VeryEven}.

  The maximal compact subgroup $\SL_2 (\bC)$ is $\SU_2$, geometrically a 3-sphere $S^3$.
  Thus, $\SL_2 (\bC)\simeq S^3$ and every $\varphi_u$ yields an element of $\pi_3 (G)=\bZ$.
  This is another role of the Dynkin index.

  Indeed, the generator of $\pi_3 (G)= \bZ$ is $\varphi_m$
  where $m\in G$ is the minimal unipotent (the exponent of the long root vector $e_\alpha \in \fg$).
  To observe this, we need to prove that the Dynkin index $\pi_3 (\varphi_m)$ is equal to 1.
  Note that the Dynkin index is computed by restricting the invariant form along the Lie algebra homomorphism
  $d\varphi_m: \fsl_2 \rightarrow \fg$ so that $\pi_3 (\varphi_m)(\cdot,\cdot)_{\fsl_2} = d\varphi_m^{\ast}(\cdot,\cdot)_{\fg}$.
  The forms are normalised on the long root vector by $(e,e)=1$ \cite{Dynkin1952,onishchik1994topology}.
  Since the root vector of the minimal $\fsl_2$ is the long root vector of $\fg$, we conclude that $\pi_3 (\varphi_m) =1$.

  Let $u\in G$ be a unipotent with Dynkin index $\nu (u)=d$. Then
  the following maps are homotopic:
  \begin{equation} \label{dynkin_hom}
  \varphi_u \simeq \varphi_m \circ \mu_d \, : \, \SL_2 (\bC)\xrightarrow{\mu_d : x\mapsto x^{d}} \SL_2 (\bC) \xrightarrow{\varphi_m }G\, .
  \end{equation}
Thus, the Dynkin index determines all the key maps in the long exact sequence~\eqref{long_es2}:

\begin{prop} \label{Dynkin_in}
If $u\in G$ is a unipotent element with Dynkin index $\nu (u)=d$, then for all $k\geq 2$
\[\pi_k(\varphi_u) = \pi_k(\varphi_m \circ \mu_d) = \pi_k(\varphi_m) \circ \pi_k(\mu_d) = d \cdot \pi_k(\varphi_m).\]
\end{prop} 

Suppose $u,v\in G$ are non-conjugate unipotent elements, with the same Dynkin index.
Then $G\rightarrow X_u$ and $G\rightarrow X_v$
are two (potentially distinct) fibrations with the same fibre.
In Section~\ref{Example}, the reader will see many examples of these with $X_u\not\simeq X_v$.


  \section{Rational Homotopy Type}\label{RaHoT}
Given a topological space $X$, we define the rational homotopy groups to be
$\pi_n(X,\bQ) = \pi(X) \otimes \bQ$ for $n\geq 2$.
The rational homotopy category is the localisation of the category of simply-connected topological spaces,
where the maps $f$, such that all $\pi_n(f,\bQ)$, 
$n\geq 2$ are invertible, become isomorphisms.

The spaces $X_u(G)$ are simply-connected and their rational homotopy types
(i.e., isomorphism classes in the rational homotopy category)
can be explicitly described.
Let $W$ be the Weyl group of $G$. Let
\[ d_1 = 2 < d_2 \leq d_r \leq \ldots \leq d_r\]
be its fundamental degrees:
these are the degrees of the generators of the invariant algebra $\bC [\fh]^W$
where $\fh$ is a Cartan subalgebra in the Lie algebra of $G$.
The following result is due to Serre.
\begin{prop}\label{typeG} \cite[Ch.~15(f)]{felix2012rational}
  In the rational homotopy category, a simply-connected simple group $G$ is isomorphic to the
  product of odd-dimensional spheres
  $$G \stackrel{\bQ}{\simeq}  \prod_{i=1}^{r} S^{2d_i - 1} \, .$$
\end{prop}
Intuitively, in the rational homotopy category, all non-trivial homomorphisms $\varphi_u$
are the first component embedding $S^3 \rightarrow \prod_{i} S^{2d_i - 1}$, while the quotient map
$G\rightarrow X_u(G)$ is the projection along the first component  $\prod_{i} S^{2d_i - 1} \rightarrow \prod_{i>1} S^{2d_i - 1}$.
We give a rigorous proof of a weaker version of this statement, sufficient for our ends.
\begin{prop} \label{typeGH}
  Let $1\neq u\in G$ be a non-trivial unipotent. In the rational homotopy category,
  $X_u(G)\stackrel{\bQ}{\simeq} \prod_{i=2}^{r} S^{2d_i - 1}$. 
\end{prop}
\begin{proof}
We will need the semifree DG-algebras
\[
\Lambda \la x_1 , \ldots , x_n \ra \coloneqq  \bQ \langle x_1 , \ldots , x_n \rangle / (x_i x_j - (-1)^{a_i a_j} x_jx_i)\, , \ 
a_i = \deg (x_i) .
\]
They are positively graded algebra with a differential $D$ of degree $-1$.
A Sullivan model for $X_u(G)$ is known:
\cite[Prop.~15.16]{felix2012rational}  
\[
\Lambda_u (G) \coloneqq
\Lambda \la x_1 , \ldots , x_r, y \ra , 
\deg (x_i) = 2d_i - 1 , \deg (y) = 2, D(y)=0, D(x_i) = \alpha_i y^{d_i-1}.
\]
The coefficients $\alpha_i$ are obtained by choosing the fundamental invariants for $G$ and $\SL_2$, defined over $\bQ$
\[
\bQ [\fh_{\bQ}]^W = \bQ [ \theta_1, \ldots , \theta_r ], \ \deg(\theta_i) = d_i, \
\bQ [\fh_{\bQ}(\SL_2)]^{W(\SL_2)} = \bQ [ \eta ],
\]
then restricting them under $d\varphi_u : \fh_{\bQ}(\SL_2) \rightarrow \fh_{\bQ}$
\[ d\varphi_u^\ast (\theta_i) = \alpha_i \eta^{d_i/2}. \]
Note that $\alpha_i=0$ for odd $d_i$ and $\alpha_1 \neq 0$. Hence, the homomorphism of the DG-algebras
\[
\Lambda_u (G) \rightarrow
( \Lambda \la x_2 , \ldots , x_r \ra , D=0) , \ 
x_1 \mapsto 0, y \mapsto 0, x_i \mapsto x_i, i\geq 2
\]
is a quasiisomorphism and the latter is the minimal Sullivan algebra for $\prod_{i=2}^{r} S^{2d_i - 1}$.
\end{proof}  

If $X_u(G){\simeq} X_v(H)$, then $X_u(G)\stackrel{\bQ}{\simeq} X_v(H)$. 
Hence, in this case, $G$ and $H$ must have the same fundamental degrees. Non-isomorphic $G$ and $H$ have the same fundamental degrees only
if they are of type $B_n$ and $C_n$. Thus, Propositions \ref{typeG} and \ref{typeGH} yield:
\begin{cor}\label{RatHomTypeCorollary} The spaces $X_u(G)$ and $X_v(H)$ have the same rational homotopy type only in the following four cases:
  \begin{enumerate}
  \item $G\cong H$ and $u=1$ and $v=1$,
  \item $G\cong H$ and $u\neq 1$ and $v \neq 1$,
  \item $G$, $H$ have types $B_n$, $C_n$ and $u=1$ and $v=1$,
    \item $G$, $H$ have types $B_n$, $C_n$ and $u\neq 1$ and $v \neq 1$.
    \end{enumerate}
\end{cor}  

\section{Quite Even Orbits}  \label{VeryEven}
Now we describe all quite even orbits in all simple $G$.
We assume that $u\neq 1$ throughout this section.
Note that the results of this section indicate that $u=1$ should be called quite even.
This disagrees with Proposition~\ref{low_pi}, quite even $u$ correspond to the spaces
$X_u(G)$ with non-trivial $\pi_2(X_u(G))$.

Recall that it is usual to talk about even nilpotent elements $e\in \fg$ \cite[Ch.~3.8]{collingwood1993nilpotent}.
We transfer this talk to unipotent elements: $u=\mbox{Exp}(e)$ is even if and only if $e$ is even.
Being even boils down to parity of the dimensions of constituent $\fsl_2$-modules.
Indeed, odd-dimensional non-trivial irreducible $\fsl_2$-modules are faithful representations of $\PSL_2 (\bC)$,
while even-dimensional irreducible $\fsl_2$-modules are faithful representations of $\SL_2 (\bC)$.
Thus $u$ is even if and only if
the $\fsl_2$-module $\varphi_u^* (\fg)$ (restriction of the adjoint representation)
is a direct sum of odd-dimensional irreducible $\fsl_2$-modules.

The main task in this section is to ascertain parities of irreducible constituents of $\varphi_u^* (V)$
for some key $\fg$-modules.

\begin{prop}
  Every quite even unipotent is even.
  If $G$ is of type $A_{2k}$, $G_2$, $F_4$, $E_6$ or $E_8$,
  all non-trivial even unipotent elements are quite even.
\end{prop}
\begin{proof}
Consider the surjection to the adjoint group $\psi : G \rightarrow G_{ad} = G/Z(G)$.
If $u$ is quite even, then
$\varphi_u (\SL_2(\bC)) = \PSL_2(\bC)$.
Hence,  $\psi(\varphi_u (\SL_2(\bC)))=\PSL_2(\bC)$, which means $u$ is even.

For the second statement, consider a non-trivial even, not quite even $u$.
Then $\varphi_u (\SL_2(\bC)) = \SL_2(\bC)$ and $\psi(\varphi_u (\SL_2(\bC)))=\PSL_2(\bC)$.
This means the centre $Z(G)$ contains 
$-I_2\in \SL_2(\bC)$.
But in all these cases, the order of $Z(G)$ is odd, so no such $u$ exists.
\end{proof}

Let $\alpha_1, \ldots , \alpha_r$ be the simple roots.
Consider the vector $[u] \coloneqq (\alpha_i (h))$ where $h$ is the semisimple element in the $\fsl_2$-triple
$e,h,f$ with $u= \mbox{Exp} (e)$.
Note that the coefficients of $[u]$ are precisely the weights of the weighted Dynkin diagram of $u$
\cite[Ch.~3.3]{collingwood1993nilpotent}.
In particular, $\alpha_i (h)\in \{0,1,2\}$ and $u$ is even if and only if all $\alpha_i (h)$ are even.

\begin{prop} \label{cirterion}
  Let $C$ be the Cartan matrix of $G$. Let $1\neq u \in G$ be a unipotent element.
Then $u$ is quite even if and only if all the coefficients of $C^{-1}[u]$ are even.
\end{prop}
\begin{proof}
  The inverse Cartan matrix
  expresses the dominant weights $\varpi_i$ from the simple roots $\alpha_i$.
  Thus, $C^{-1}[u]=(\varpi_i (h))$.
  This vector is even if and only if all the weights 
  of the restriction $\varphi_u^{\ast}(V)$ for any finite-dimensional $G$-module $V$ are even. This is equivalent to $\im (\varphi_u)=\PSL_2(\bC)$
  and $u$ being quite even.
\end{proof}

Proposition~\ref{cirterion}
could be used for any particular $u$. The next corollary is verified by this direct calculation for each $u$.
We follow the standard labels for the unipotent classes \cite{Carter93,collingwood1993nilpotent}.

\begin{cor}
  Let $G$ be of type $E_7$. These are its quite even classes:
  \[
  A_2,
  2A_2,
  D_4(a_1),
  D_4,
  A_4, 
  A_4+A_2,
  E_6(a_3),
  D_5,
  A_6,
  E_6(a_1),
  E_6.
  \]
These are its non-trivial even, but not quite even classes:
\[
(3A_1)^{\prime\prime},
A_2+3A_1,
(A_3+A_1)^{\prime\prime},
A_3+A_2+A_1,
(A_5)^{\prime\prime},
\]
\[
D_5(a_1)+A_1,
E_7(a_5),
E_7(a_4),E_7(a_3),E_7(a_2),E_7(a_1),E_7.
\]
\end{cor}

We finish by describing quite even (and also even) elements in the classical types in terms of partitions.
Suppose $G$ is one of the groups $\SL_n (\bC)$, $\Spi_n (\bC)$ or $\Sp_n (\bC)$. Restrict the natural representation
of $G$ on $V=\bC^n$ to $\SL_2(\bC)$, along the map $\varphi_u$ and record the dimensions of irreducible constituents
$(d_1\geq d_2 \geq \ldots)$. This is the partition $p(u) \in \mP(n)$, associated to $u$ \cite[5.1.7]{collingwood1993nilpotent}.

Recall the natural restrictions on the parities of $d_i$ in $p(u)$. 
They come from the fact that odd-dimensional representations of $\SL_2 (\bC)$ are orthogonal,
while even-dimensional representations are symplectic.
In type $C_r$, $V$ carries a symplectic form, whose restrictions to odd-dimensional constituents of $\phi_u^{\ast} (V)$ must be zero.
Thus, such constituents must come in dual isotropic pairs of spaces. Thus, all odd $d_i$ must appear an even number of times.

Similarly, in types $B_r$ and $D_r$, $V$ carries an orthogonal form, whose restrictions to even-dimensional constituents of $\phi_u^{\ast} (V)$ must be zero.
Thus, such constituents must come in dual isotropic pairs of spaces and all even $d_i$ must appear an even number of times.

\begin{theorem}\label{QuiteEvenClassical}
  Consider $G$ of classical type and a unipotent element $u\in G$, $u\neq 1$ with the corresponding partition $p(u)=(d_i)\in \mP(n)$.
  \begin{enumerate}
  \item $u$ is even if and only if all $d_i$ have the same parity (all even or all odd).
  \item If $G$ is of type $A_r$ or $C_r$, then $u$ is quite even if and only if all $d_i$ are odd. 
  \item If $G$ is of type $B_r$ or $D_r$, then $u$ is quite even if and only if the product $\prod_i d_i \equiv \pm 1 \mod 8$.
  \end{enumerate}  
\end{theorem}
\begin{proof}
  The central element $-I_2\in \SL_2 (\bC)$ acts as  $(-1)^{d+1}$ on the $d$-dimensional irreducible $\fsl_2$-module.
  Thus, all the questions of $u$ being (quite) even are reduced
  to a certain $\fsl_2$-module having only odd-dimensional constituents.
  For instance, Statement~(2) is immediate since the standard representation $V=\bC^n$ is a faithful $G$-module in the types $A_r$ or $C_r$.

  In all the cases the adjoint $G$-module is a constituent of $V\otimes V^\ast$. This proves the ``if'' part of Statement~(1):
  $-I_{2}$ acts as $1$ on  $V\otimes V^\ast$.

  The ``only if'' part of Statement~(1) is clear too.
  The  adjoint $G$-module can be recovered from $V$ as
  $$
  \fg_{ad} = \frac{V\otimes V^\ast}{ \bC I_V} \ (\mbox{type } A), \
  \fg_{ad} = \Lambda^2 V \ (\mbox{types } B, D), \
  \fg_{ad} = S^2 V  \ (\mbox{type } C). 
  $$
  Suppose $u$ is even so that $-I_2$ acts on $\fg_{ad}$ as $1$.
  Consider two irreducible constituents $U_1, U_2$ of dimensions $d$ and $d^\prime$ of $\varphi_u^{\ast}V$.
  If $dd^\prime>1$, 
  some constituents $U_1\otimes U_2 \cong U_1\otimes U_2^\ast$ inevitably appear in $\fg_{ad}$.
  But there $-I_2$ acts as $(-1)^{d+d'+2}$, proving that $d$ and $d'$ have the same parity.

  In the extreme case of $d=d'=1$, we can take any other constituent $U_3$ of degree $\tilde{d}$.
  The argument above shows that $\tilde{d}$ is odd so that all $d_i$ are odd.

  Statement~(3) requires the image of $-I_2$ in a faithful representation of $G$.
  In type $B_r$, the image in the the spinor representation is computed  in Lemma~\ref{typeB}.
  In type $D_r$, the image in the sum of the semispinor representations is computed in Lemma~\ref{typeD}.
  These two lemmas complete the proof.
\end{proof}

In types $A_r$ and $D_r$, there are intermediate groups   
$\psi : G \rightarrow G/A$, $A\leq Z(G)$, 
which are neither simply-connected, nor adjoint.
For completeness, let us quickly address the conditions when the image of $\SL_2 (\bC)$ in these groups is $\PSL_2 (\bC)$.
Clearly, $u$ being  quite even is sufficient and $u$ being even is necessary. This leaves the following cases to address.
\begin{cor} \label{remaining_groups}
The following statements hold for an even, not quite even unipotent element $u\in G$, $u\neq 1$ with the corresponding partition $p(u)=(d_i)\in \mP(n)$.
\begin{enumerate}
  \item Let $G$ be of type $A_{2k-1}$, $A= \langle e^{2\pi i /s} I_{2k} \rangle$ where $s$ divides $2k$.
    Then $\psi (\varphi_u (\SL_2 (\bC))) = \PSL_2 (\bC)$
    if and only if
    $m=2k/s$ is even.
  \item Let $G$ be of type $D_{r}$, $G/A= \SO_{2r} (\bC)$.
    Then $\psi (\varphi_u (\SL_2 (\bC))) = \PSL_2 (\bC)$
    if and only if
      all $d_i$ are odd. 
  \item Let $G$ be of type $D_{2k}$, $G/A= \SSpi_{4k} (\bC)$, one of the semispin groups.
    Then $\psi (\varphi_u (\SL_2 (\bC))) = \PSL_2 (\bC)$
    if and only if
     the product $\prod_i d_i \equiv \pm 1 \mod 8$.
  \end{enumerate}
\end{cor}
\begin{proof}
  Let $\varpi_1, \ldots, \varpi_r$ be the fundamental weights, so that $V=V(\varpi_1)$. Under our assumption on $u$, $-I_2$ acts as $-1$ on $V$.

  A faithful representation of $\SL_{2k}(\bC)/A$ is $V(\varpi_m) = \Lambda^m (V)$. 
  Statement~(1) immediately follows because $-I_2$ acts as $(-1)^m$ on $\Lambda^m (V)$. 

        Statement~(2)
        holds because the $d_i$ are dimensions of the irreducible constituents of $\phi_u^{\ast}(V)$, which is a faithful representation of $\SO_{2r} (\bC)$.

              Statement~(3) is essentially Lemma~\ref{typeD} below. 
\end{proof}

To finish the proof in type $B_r$, we need to compute the action of $-I_2\in\SL_2(\bC)$ on the spin representation.
We briefly recall its construction, following \cite[Section 20.1]{fulton2013representation}. 
Let $V=\bC^n$ where $n=2r+1$.
Following \cite[Section 18.1]{fulton2013representation}, we fix a basis $e_1,\ldots,e_{n}$ of $V$ and define a bilinear form $Q$ via
\[Q(e_i,e_{r+i}) = Q(e_{r+i},e_{i}) =1,\hspace{3ex} Q(e_i,e_j) = 0 \text{ otherwise}.\]
The corresponding Clifford algebra (with respect to $Q$) 
\[C(Q) = C(V,Q) := T^{\bullet}(V)/\left(a\otimes b + b\otimes a - Q(a,b)\right)\]
admits a $\mathbb{Z}/2\mathbb{Z}$ grading $C(Q) = C(Q)_{0} \oplus C(Q)_{1}$. We have an embedding of $\mathfrak{so}(V)$ into $C(Q)_{0}$ via:
\begin{gather*}
    \mathfrak{so}(V) \cong \Lambda^2(V) \longrightarrow C(V)^{0} \\
    2(E_{i,j} - E_{r+j,r+i}) \longmapsto e_i \wedge e_{r+j} \longmapsto  \left(e_ie_{r+j} - \delta_{i,j}\right) 
\end{gather*}
where $E_{i.j}$ denotes the matrix with a $1$ in position $(i,j)$ and $0$ elsewhere. Note that
we write $e_ie_{r+j} \in C(Q)_{0}$ instead of $e_i \otimes e_{r+j}+(a\otimes b + \ldots)$.

Now write $V = W \oplus W' \oplus U$
where $W = \langle e_1,\ldots,e_r \rangle$, $W' = \langle e_{r+1},\ldots,e_{2r}\rangle $ and $U = \langle e_{2r+1} \rangle$. Observe that $W$ and $W'$ are isotropic subspaces with respect to $Q$. Moreover, we have an identification of $W'$ and $W^*$ via $e_{r+i} \mapsto e_{i}^* = \frac{1}{2}Q(e_{r+i},-)$. Now, one can show that we have an isomorphism \cite[Eq 20.18]{fulton2013representation}
\begin{gather*}
C(Q)_{0} \cong \text{End}(\Lambda^\bullet W) \\
e_i \longmapsto L_{e_i}: (w_1 \wedge \ldots \wedge w_s \longmapsto e_i \wedge w_1 \wedge \ldots \wedge w_s) \\
e_{r+i} \longmapsto D_{e_{i}}^*: \left(w_1 \wedge \ldots \wedge w_s \longmapsto \sum (-1)^{i-1}2Q(w_i,e_i)w_1 \wedge \ldots \wedge \widehat{w_i} \wedge \ldots \wedge w_s\right) \\
e_{2r+1} \longmapsto \big(w_1 \wedge \ldots \wedge w_r \longmapsto (-w_1) \wedge \ldots \wedge (- w_r)\big).
\end{gather*}  
This, alongside our embedding $\mathfrak{so}(V) \subset C(Q)_{0}$ gives a representation of $\mathfrak{so}(V)$ on $\Lambda^\bullet W$, known as the spin representation.
We are ready to calculate
the image of $h \in \mathfrak{sl}(2)$ (from the corresponding $\fsl_2$-triple)
inside $\text{End}(\Lambda^\bullet W)$.


\begin{lem} \label{typeB}
  Let $p(u)=(d_i)\in \mP(n)$ be a partition of $n=2r+1$, with all $d_i$ odd and
  the corresponding map $d\varphi_u: \mathfrak{sl}(2) \longrightarrow \mathfrak{so}(n)$.
  Then the weights appearing in $\varphi_u^{\ast} \Lambda^\bullet W$
  are all even or all odd. Moreover, the weights are all even if and only if $\displaystyle\,\prod_{i=1}^k d_i \equiv \pm 1 \mod 8.$  
\end{lem}

\begin{proof}
  Let $h = \begin{pmatrix} 1 & 0 \\ 0 & -1 \end{pmatrix}\in\fsl_2$.
  Write our partition $[d_1,\dots d_k]$ as $[2c_1+1,\dots,2c_k+1]$.
  Consider the diagonal matrix
  $$
  D = \text{diag}(2c_1, 2c_1 -2 , \ldots,4,2,2c_2, \ldots,2c_k,\ldots 2,\underbrace{0,\ldots,0}_{\frac{k-1}{2} \text{ zeroes}}) \in
  M_{r\times r}(\bC).
  $$
  Then $\varphi_u (h) = \text{diag}(D,-D,0)$ \cite[5.2.4]{collingwood1993nilpotent}.
  Tracing through the maps 
\[\mathfrak{so}(V) \cong \Lambda^2 V \longrightarrow C(Q)_{0} \longrightarrow \text{End}(\Lambda^\bullet W)\]
associates to $h$ the endomorphism 
\[ \sum_{i=1}^{k} \sum_{j=1}^{c_i} (c_i +1-j)\left(2L_{e_{c_1+\dots+c_{i-1}+j}} \circ D_{e_{c_1+\dots+c_{i-1}+j}^*} -\textup{Id}\right) \in \text{End}(\Lambda^\bullet W).\]
Now we have for $e_I := e_{i_1} \wedge \dots \wedge e_{i_k} \in \Lambda^\bullet W$: \cite[within proof of 20.15]{fulton2013representation}
\[\left(2L_{e_i} \circ D_{e_{i}^*} - Id\right)(e_I) = \begin{cases} e_I & i \in I \\ -e_I & i \not\in I \end{cases}.\]
It follows that each $e_I$, $I \in 2^{\{1,\dots,n\}}$ (the power set of $\{1,\dots,n\}$) is a weight vector.
The weights of $h$ correspond to the $2^n$ ways to assign signs to the sum 
\[\pm c_1 \pm (c_1 - 1) \pm \dots \pm 1 \pm c_2 \pm \dots \pm 1 \pm \dots \pm c_k \pm \dots \pm 1 \underbrace{\pm 0 \pm \dots \pm 0}_{\frac{k-1}{2} \text{ zeroes}}.\]
In particular, they are all even or all odd: changing a single sign results in adding or subtracting of twice a single summand.

To determine whether it is even or odd, we can choose all signs to be even. Then
\begin{align*}
\text{All weights are even} &\iff \sum_{i=1}^k \frac{c_i(c_i+1)}{2} \equiv 0 \mod 2 \\
&\iff \left|\left\{c_i : \frac{c_i(c_i+1)}{2} \equiv 1 \mod 2 \right\}\right| \text{ is odd}\\
&\iff \left|\left\{c_i : c_i \equiv 1,2 \mod 4 \right\}\right| \text{ is odd} \\
&\iff \left|\left\{d_i : d_i \equiv 3,5 \mod 8 \right\}\right| \text{ is odd} \\
&\iff \prod_{i=1}^k d_i = \pm 1 \mod 8 
\end{align*}
\end{proof}

Let us finish the proof in type $D_r$.
Similarly, to type $B_r$
we need to compute 
the action of $h$ on the half-spin representation
$\Lambda^{2\bullet} W$ or 
$\Lambda^{2\bullet +1}W$ where $W=\bC^r$ is a lagrangian subspace in $V=\bC^{2r}$.

\begin{lem} \label{typeD}
  Let $p(u)=(d_i)\in \mP(n)$ be a partition of $n=2r$,
  with all $d_i$ odd and
  corresponding map
  $d\varphi_u: \mathfrak{sl}(2) \longrightarrow \mathfrak{so}(n)$.
  Then the weights appearing in $\varphi_u^{\ast} \Lambda^{2\bullet} W$ (or in $\varphi_u^{\ast} \Lambda^{2\bullet+1} W$)
  are all even or all odd. Moreover, the weights are all even if and only if $\displaystyle\,\prod_{i=1}^k d_i \equiv \pm 1 \mod 8.$
\end{lem}
\begin{proof}
  The proof goes verbatim to the type $B_r$ proof.
  The key difference in this case is that
  $$
  C(Q)_{0} \cong \text{End}(\Lambda^{2\bullet}W) \oplus \text{End}(\Lambda^{2\bullet +1}W). \ \ \mbox{\cite[20.13]{fulton2013representation}}
    $$
    This yields the two semispin representations of $\mathfrak{so}(V)$.
    Following in the footsteps of  the type $B_r$ proof, we observe that the weights of $h$
    on $\varphi_u^{\ast}\Lambda^{2\bullet} W$ and  $\varphi_u^{\ast} \Lambda^{2\bullet+1} W$
    are the $2^{r-1}$ ways to assign signs to the sum
\[\pm c_1 \pm (c_1 - 1) \pm \dots \pm 1 \pm c_2 \pm \dots \pm 1 \pm \dots \pm c_k \pm \dots \pm 1 \underbrace{\pm 0 \pm \dots \pm 0}_{\frac{k-1}{2} \text{ zeroes}}\]
such that an even (resp. odd) number of the signs are $'+'$.
In particular, the finale of the type $B_r$ proof works here too.
\end{proof}

%

\section{Higher Homotopy Groups}  \label{HiHoGr}

As seen in Corollary \ref{RatHomTypeCorollary}, rational homotopy theory do not differentiate between spaces
$X_u(\Spi_{2r+1})$ and $X_v(\Sp_{2r})$ -- we skip $\bC$ in this section.
However, this can be done by considering several other homotopy groups.
Recall  that, by Bott periodicity, we know the following homotopy groups:
\begin{align*}
\pi_3(H) = \bZ, \pi_4(H) = \bZ/2, \ & \pi_5(H) = \bZ/2, \pi_6(H) = \bZ/12, \\
\pi_3(\Spi_{2r+1} ) = \bZ, \pi_4(\Spi_{2r+1} ) = 0, \ & \pi_5(\Spi_{2r+1} ) = 0, \pi_6(\Spi_{2r+1} ) = 0, \\
\pi_3(\Sp_{2r} ) = \bZ, \pi_4(\Sp_{2r} ) = \bZ/2, \ & \pi_5(\Sp_{2r} ) = \bZ/2, \pi_6(\Sp_{2r} ) = 0.
\end{align*}
where, as before, $H$ is $\SL_2$ or $\PSL_2$.

Knowing these, the long exact sequence~\eqref{long_es} yields some information about homotopy groups of $X_u(G)$, summarised in Table \ref{low-homotopy-BC-table}
\begin{table}[h!]
\renewcommand{\arraystretch}{1.5} 
\centering
\begin{tabular}{ |P{0.8cm}| P{1.8cm}| P{8.5cm} | }
 \hline
 &$\faktor{\Spi_{2r+1} }{H}$&$\faktor{\Sp_{2r} }{H}$\\
 \hline
 $\pi_3$&$\bZ/d\bZ$&$\bZ/d\bZ$\\
 \hline
 $\pi_4$&$0$&$\begin{cases} \bZ/2\bZ & \text{ if } b_4 = 0 \\ 0 & \text{ if } b_4 = 1 \end{cases}$\\
 \hline
 $\pi_5$&$\bZ/2\bZ$&$\begin{cases} \bZ/2\bZ \oplus \bZ/2\bZ \text{ or } \bZ/4\bZ & \text{ if } (b_4,b_5) = (1,1) \\ \bZ/2\bZ & \text{ if } (b_4,b_5) = (1,0) \text{ or } (0,1) \\ 0 & \text{ if } (b_4,b_5) = (0,0) \end{cases}$\\
 \hline
 $\pi_6$&$\bZ/2\bZ$&$\begin{cases} \bZ/2\bZ & \text{ if } b_5 = 0 \\ 0 & \text{ if } b_5 = 1 \end{cases}$\\
 \hline
\end{tabular}
\caption{Low degree homotopy groups of $X_u(G)$}
\label{low-homotopy-BC-table}
\end{table}
where the values $d$, $b_4$ and $b_5$ are given as follows: 
\begin{itemize}
    \item $d$ is the Dynkin index, i.e., the map $d: \pi_3(H) =\bZ \rightarrow \pi_3(G) =\bZ$,
    \item $b_4$ represents the map $b_4: \pi_4(H) = \bZ/2\bZ \rightarrow \pi_4(\Sp_{2r} ) = \bZ/2\bZ$, 
    \item $b_5$ represents the map $b_5: \pi_5(H) = \bZ/2\bZ \rightarrow \pi_5(\Sp_{2r} ) = \bZ/2\bZ$. 
\end{itemize}
By Proposition~\ref{Dynkin_in}, 
these three values are not independent.
We can say a bit more if we can compute $\pi_\ast (\varphi_m)$.
The proof of the next result was explained to us by Michael Albanese \cite{stackex5}.
\begin{prop}
If $G$ is of type $C_{r}$, $r\geq 1$, then
$\pi_k(\varphi_m) = 1$ for $k \leq 5$.
\end{prop} 
\begin{proof}
  We begin by noting that $\SL_2 \simeq \Sp(1)$ and $G\simeq \Sp(r)$ (these are the hyperunitary groups).
  Hence, 
  the induced map $\pi_*(\SL_2) \xrightarrow{\varphi_m} \pi_*(\Sp_{2r} )$ may be thought of as a map
  $\pi_*(\Sp(1)) \rightarrow \pi_*(\Sp(r))$.
  Also, note that since $\pi_4(\Sp_{2r} ) = \pi_5(\Sp_{2r} ) = \bZ/2$,
  showing that $\pi_4(\varphi_m) = \pi_5(\varphi_m) = 1$ is equivalent to showing that these maps are isomorphisms.

  Now, 
  the quaternionic unitary group $\Sp(r)$
  acts transitively on the sphere $S^{4r-1} \subset \bH^r$, with stabiliser $\Sp(r-1)$.
  Choosing a base point $x = (0,0,\ldots,0,1)^T \in S^{4r-1} \subset \bH^r$ gives us the maps
\[
\rho: \Sp(r) \rightarrow S^{4r-1}, \
\rho(A) = Ax, \ 
\iota_r: \Sp(r-1) \rightarrow \Sp(r), \
\iota_r(A) = \begin{pmatrix} A & 0 \\ 0 & I_2
\end{pmatrix}
\]
that form 
a fibre bundle
\[
\Sp(r-1) \rightarrow \Sp(r) \xrightarrow{\rho} S^{4r-1} \, ,
\]
which, in its turn, induces a long exact sequence in homotopy groups 
\[ \ldots \to \pi_{k+1}(S^{4r-1}) \to \pi_k(\Sp(r-1)) \xrightarrow{(\iota_r)_*} \pi_k(\Sp(r)) \xrightarrow{\rho_*} \pi_k(S^{4r-1}) \to \ldots\]
Notice that $\pi_{k}(S^{4r-1}) = \pi_{k+1}(S^{4r-1}) = 0$ for $k \leq 4r-3$ and so it immediately follows that
${\pi_k (\iota_r): \pi_k(\Sp(r-1)) \rightarrow \pi_k(\Sp(r))}$ is an isomorphism for $k \leq 4r - 3$. 

Now define $\iota \coloneqq \iota_r \circ \ldots \circ \iota_2$. We observe that
\[
\iota : \Sp(1) \rightarrow \Sp(r), \
\iota (A) = \begin{pmatrix} A & 0 \\ 0 & I_{2r-2} \end{pmatrix}, \ 
\pi_k (\iota) = \pi_k (\iota_r) \circ \ldots \circ \pi_k (\iota_2)\, .
\]
In particular, $\iota \simeq \varphi_m$.
Observe that for $m \geq 2$ we have $4 \leq 5 \leq 4m - 3$ and so 
\[\pi_4 (\iota_m): \pi_4(\Sp(m-1)) \rightarrow \pi_4(\Sp(m)) \mbox{ and } \pi_5 (\iota_m): \pi_5(\Sp(m-1)) \rightarrow \pi_5(\Sp(m))\]
are isomorphisms. Since a composition of isomorphisms is also an isomorphism we have that
\[\pi_4 (\varphi_m) : \pi_4(\Sp(1)) \rightarrow \pi_4(\Sp(r)) \text{ and } \pi_5 (\varphi_m) : \pi_5(\Sp(1)) \rightarrow \pi_5(\Sp(r))\]
are isomorphisms, as required.
\end{proof}

\begin{cor} \label{BC_separate}
For all unipotent elements $u \in \Spi_{2r+1} $, $v \in \Sp_{2r} $, $n \geq 3$, we have $X_u(\Spi_{2r+1} ) \not\simeq X_v(\Sp_{2r} )$.
\end{cor}

\begin{proof}
 Looking at the table above, we have $b_4 \equiv b_5 \equiv d \,(\text{mod } 2)$ by the proposition.
It follows that $\pi_5\left(X_u(\Spi_{2r+1} )\right) \neq \pi_5\left(X_v(\Sp_{2r} )\right)$ and 
$X_u(\Spi_{2r+1} ) \not\simeq X_v(\Sp_{2r} ).$
\end{proof}

Notice that Corollaries \ref{BC_separate} and \ref{RatHomTypeCorollary} together prove 
Theorem~\ref{main_theorem}.

\section{K-Theory}\label{KTheory}
Recall that if $X\simeq Y$  then we have an isomorphism of the topological K-theories $K(X) \cong K(Y)$ (as graded $\lambda$-rings). 
Topological K-theory is a powerful invariant for working with the spaces $X_u=X_u(G)$. Let us describe it.

Recall that 
the representation ring
of a simple simply-connected group $G$ of rank $r$ is polynomial:
$$
R(G) = \mathbb{Z}[y_1,\dots,y_r] , \ \mbox{ where } \ y_i = [V(\varpi_i)]
$$
and $\varpi_i$ is the $i$-th fundamental weight. 
Similarly,
$$
R(\SL_2(\mathbb{C})) = \mathbb{Z}[x] \geq R(\PSL_2 (\bC)) = \mathbb{Z}[x'], \ \mbox{ where } \ x'=x^2-1
$$
so that  $x  = [V(\varpi_1)]$ is the class of the 2-dimensional simple representation and
$x'  = [V(2\varpi_1)]$ is that of the 3-dimensional one.
We will suppress the prime and just write $R(H) =\mathbb{Z}[x]$ for the representation ring of $H$ in both cases with $d\in\{2,3\}$ being the dimension
of the module represented by $x$.
It is well known that the $G$-equivariant K-theory of a homogeneous space $G/H$ is related to the representation ring of $H$, i.e.,
$K_{G}^0\left( X_u \right) = R(H)$. \cite[Ex. 2.ii]{segal1968equivariant}

The full K-theory of $X_u$ is attainable via the Hodgkin spectral sequence, a version of the Kunneth formula in equivariant K-theory \cite{hodgkin1968equivariant}
(cf. \cite{minami1975k}). The multiplicative 2-periodic spectral sequence 
\[E_2^{*,0} \coloneqq \Tor_{R(G)}^{*}\left(K_G^0(G),K_G^0\left(X_u \right)\right) = \Tor_{R(G)}^{*}\left(\bZ,R(H)\right) \implies K^* ( X_u )\]
collapses, that is, all differentials $\bd_k$ vanish for $k \geq 2$. 
The $R(G)$-module structures are given by
\begin{equation} \label{maps}
\bZ [y_1, \ldots, y_r] \rightarrow \bZ[x], \ y_i \mapsto \overline{y_i}
\ \ \mbox{ and } \ \ 
\bZ [y_1, \ldots, y_r] \rightarrow \bZ, \ y_i \mapsto d_i
\end{equation}
where $d_i = \dim (V(\varpi_i))$ and 
$\overline{y_i}= [\varphi_u^{\ast} V(\varpi_i)]$. The collapse of the multiplicative spectral sequence yields the next result.


\begin{prop}\label{KT_presentation}
There is an isomorphism of abelian groups 
\[K^0 (X_u(G)) \cong \Tor_{R(G)}^{2\bullet }\left(\bZ,R(H)\right) = \bigoplus_{k=0} \Tor_{\bZ[y_1, \ldots, y_r]}^{2k}\left(\bZ,\bZ[x]\right)\]
yielding an isomorphism of graded rings
\[\textup{Gr}\left(K^0 (X_u(G))\right) \cong \Tor_{R(G)}^{2\bullet }\left(\bZ,R(H)\right),\] where $\textup{Gr}(-)$ denotes the associated graded ring. Furthermore, there is an isomorphism of abelian groups 
\[K^1 (X_u(G)) \cong \Tor_{R(G)}^{1+2\bullet}\left(\bZ,R(H) \right) = \bigoplus_{k=0} \Tor_{\bZ[y_1, \ldots, y_r]}^{2k+1}\left(\bZ,\bZ[x]\right) \]
yielding an isomorphism of $\textup{Gr}(K^0(X_u(G)))$-modules.
\end{prop}

Notice that Proposition~\ref{KT_presentation} equips $\textup{Gr}(K^0(X_u))$
with a structure of a graded $\bZ[x]$-algebra. Let us note its grade $0$ piece
\begin{equation} \label{Iu_def}
K^0(X_u)_0 \cong R(H) \otimes_{R(G)} \bZ = \bZ[x] / I_u \ \mbox{ where } \ I_u = (\overline{y_1}-d_1, \ldots ,  \overline{y_r}-d_r ) \, .
\end{equation}
If $X_u\simeq X_v$,
we get two graded algebra structures on the same ring. Yet we can relate the zero-degree pieces (cf. Theorem~\ref{2nd_thm})!

To compute in Proposition~\ref{KT_presentation}, we need a DG-algebra resolution.
Since
$\bZ \cong \bZ[\underline{y}] /(y_1-d_1,\dots,y_r-d_r)$
as $R(G)$-modules, a convenient resolution
is given by the Koszul complex (cf. Section~\ref{code} for codes) 
$\; \bK_\bullet (y_1-d_1, \ldots, y_r-d_r)$
\[
0 \longrightarrow \bigwedge^r \mathbb{Z}[\underline{y}]^r \longrightarrow \bigwedge^{r-1} \mathbb{Z}[\underline{y}]^r\longrightarrow \dots \longrightarrow \mathbb{Z}[\underline{y}]^r \xrightarrow{(y_i-d_i)_{i=1}^r}\mathbb{Z}[\underline{y}] \longrightarrow 0.\]

Tensoring with $\bZ[x]$ gives us a Koszul complex over $\bZ[x]$
$$
\bK_\bullet (y_1-d_1, \ldots , y_r-d_r) \otimes_{\mathbb{Z}[\underline{y}]} \mathbb{Z}[x] =
\bK_\bullet ( \overline{y_1}-d_1, \ldots ,  \overline{y_r}-d_r) 
$$
that, in turn, yields the master formula
\begin{equation}\label{computeTor}
  \Tor_{R(G)}^{l}\left(\bZ,R(H)\right) = H_l \left( \bK_\bullet ( \overline{y_1}-d_1, \ldots ,  \overline{y_r}-d_r) \right)
\  \mbox{ for all } l \, .
\end{equation}
The homology of the Koszul complex over $\bZ [x]$ is computable but not to the extent of the next lemma.
A careful reader may observe that it also works (after careful adjustment) over a Dedekind domain.
\begin{lem}\label{DedkindKoszullemma}
  Suppose $R$ is a PID.
  Let $\bK_\bullet = \bK_\bullet (a_1,\dots,a_m)$, $a_i \in R$ denote the Koszul complex over $R$
\[0 \rightarrow R \xrightarrow{\bd_m} R^m \xrightarrow{\bd_{m-1}} R^{\binom{m}{m-2}} \xrightarrow{\bd_{m-2}} \dots \longrightarrow R^{\binom{m}{2}} \xrightarrow{\bd_2} R^m \xrightarrow{\bd_1} R \rightarrow 0 \]
Setting $\bd_0 = \bd_{n+1} = 0$ and $a=\gcd(a_1,\ldots,a_m)$, we have
\[ H_i (\bK_\bullet) = \left(\faktor{R}{\left(a\right)}\right)^{\binom{m-1}{i}} \text{ for all }0 \leq i \leq m.\]
\end{lem}
\begin{proof}
%
Pick a prime $p\in R$. Consider the localised Koszul complex
  $$
  \bK_{(p)} (a_1,\dots,a_m) = \bK (a_1,\dots,a_m) \otimes_R R_{(p)} \, .
  $$
  For any units $u_i \in R_{(p)}$, the complexes
  $\bK_{(p)} (a_1,\ldots,a_m)$ and $\bK_{(p)} (u_1a_1,\dots,u_ma_m)$ are quasi-isomorphic via the obvious chain map.
  In particular, they share the same homology. Thus, by choosing the $u_i$ such that $u_ia_i = p^{k_i}$ for some $k_i$,
  and assuming (without loss of generality) that $k_1 \leq \dots \leq k_m$, we conclude that
  $$
  H_\bullet\big( \bK_{(p)} (a_1,\ldots,a_m) \big) = H_\bullet\big(\bK_{(p)} ( p^{k_1},\ldots,p^{k_m})\big)
  \, .
  $$
  By inspection of the differentials $\bd_i$ of $\bK_{(p)} ( p^{k_1},\ldots,p^{k_m})$,
    we see that $\im(\bd_{i+1})$ is given by $R_{(p)}$-linear combinations of those columns $c$ that contain a $p^{k_1}$ term. Moreover, since $p^{k_1}$ divides each $p^{k_j}, j \geq 1$, we have ${p^{-k_1}}\cdot c \in {R_{(p)}}^{\binom{n}{i}}$ for each such column $c$. In particular, we have ${p^{-k_1}}\cdot c \in \ker(\bd_i)$ for all such $c$. Furthermore, any element of the kernel is a linear combination of these scaled columns. Now, there are precisely $\binom{m-1}{i}$ of these columns, namely the number of ways to choose exactly $i$ elements without replacement from the set $\{p^{k_2},\dots,p^{k_m}\}$. 
Combining both observations, we conclude that
\[H_i (\bK_{(p)}) = 
\bigoplus_{i=1}^{\binom{m-1}{i}}\, \faktor{\mathbb{Z}_{(p)}}{p^{k_1}\mathbb{Z}_{(p)}} =
\left( \faktor{\mathbb{Z}_{(p)}}{\gcd (p^{k_1},\dots,p^{k_m})} \right)^{\binom{m-1}{i}} .\]
Let $P$ be the set of all primes in $p\in R$ such that $p$ divides one of the $a_i$.
If $p\not\in P$, $k_1=0$ and the localisation at such $p$ ``kills'' homology of $\bK_\bullet$. 
A direct sum of localisations at $p\in P$ gives us the required result.
\end{proof}

Recall that the ideal $I_u$ is defined in \eqref{Iu_def}. We are ready for the main result of the section.
\begin{theorem} \label{2nd_thm}
If $X_u\simeq X_v$, then $I_u=I_v$.
\end{theorem}
\begin{proof}
  Notice that $u$ and $v$ are both quite even or not quite even. This means we can use the same $d\in\{2,3\}$ for the two realisations of
  $K_0(X_u)$ as the homology of the Koszul complexes
  \[K_0(X_u) \cong H_{2\bullet} (\bK (u)) \cong H_{2\bullet} (\bK (v))\, .\]
  Moreover, we know that
  \[H_{0} (\bK (u)) = \bZ[x] / I_u,  \ H_{0} (\bK (v)) = \bZ[x] / I_v \ \mbox{ and } \ \sqrt{I_u} = \sqrt{I_v} = (x-d).\]
  The latter follows from the fact that $x-d$ is a nilpotent element in $K(X_u)$.
  Indeed, $x-d$ belongs to the reduced K-theory $\widetilde{K} (X_u)$, the kernel of the map $K(X_u) \rightarrow \bZ$.
  As observed before Proposition~\ref{Xu_dim},  $X_u\simeq G_c/H_c$ and the latter is compact, so can be covered by $m$ open
  contractible subsets. It is a standard topological fact that $\widetilde{K} (X_u)$ is $m$-nilpotent \cite[Ex. 2.13]{HatcherVB}.

  Let $\bF=\bZ / (p)$. In $\bF[x]$ the reduced ideals $\overline{I_u}$ and $\overline{I_v}$
  are still contained in $(x-d)$, hence, $\overline{I_u}=(x-d)^a$ and $\overline{I_v}=(x-d)^b$
  for some positive integers $a$ and $b$.

  Now compare homology of Koszul complexes modulo a prime $p$. By Lemma~\ref{DedkindKoszullemma},
  \[H_{2\bullet} (\bK(u)\otimes_{\bZ} \bF) \cong \Lambda^{2\bullet} \left(\dfrac{\bF[x]}{(x-d)^a}\right)^{r-1}, \
  H_{2\bullet} (\bK(v)\otimes_{\bZ} \bF) \cong \Lambda^{2\bullet} \left(\dfrac{\bF[x]}{(x-d)^b}\right)^{r-1} . \]
  By the universal coefficients theorem, both reduced Koszul complexes fit into the same short exact sequence
  of vector spaces
\[
0 \rightarrow
H_{2\bullet} (\bK)\otimes_{\bZ} \bF \rightarrow
H_{2\bullet} (\bK \otimes_{\bZ} \bF) \rightarrow
\mbox{Tor}_1^{\bZ}(H_{-1+2\bullet} (\bK) , \bF)
\rightarrow 0\]  
that can be rewritten with $c\in \{a,b\}$ as 
\[
0 \rightarrow
K^0 (X_u)\otimes \bF \rightarrow
\Lambda^{2\bullet} \left(\dfrac{\bF[x]}{(x-d)^c}\right)^{r-1}
\rightarrow
p\mbox{\,-Tor} ( K^1 (X_u))
\rightarrow 0\, . \]
Dimension counting proves that $a=b$. Thus, $\overline{I_u} = \overline{I_v}$ for every prime $p$.

It remains to show that this implies $I_u = I_v$.
Consider the quotient $Q_u \coloneqq (I_u+I_v)/I_u$.
We will show that $Q_u = 0$. Analogously, $(I_u+I_v)/I_v=0$ and $I_u = I_v$, as required.

Since $\overline{I_u} = \overline{I_v}$ for each prime $p$,
we have $Q_u \otimes_{\bZ} \mathbb{Z}/p\mathbb{Z} = 0$ for all $p$.
Furthermore, $I_u + I_v\subseteq \bZ[x]$ has no $p$-torsion.
Looking at the map $\Tor_1(I_u+I_v,\mathbb{Z}/p\mathbb{Z}) \rightarrow \Tor_1(Q_u,\mathbb{Z}/p\mathbb{Z})$,
we conclude that $\Tor_1(Q_u,\mathbb{Z}/p\mathbb{Z}) = 0$.
Tensoring a projective resolution of $\mathbb{Z}/p\mathbb{Z}$ with $Q_u$ gives us that the sequence 
\[ 0 \longrightarrow Q_u \otimes \mathbb{Z} \xrightarrow{p} Q_u \otimes \mathbb{Z} \longrightarrow \underbrace{Q_u \otimes \mathbb{Z}/p\mathbb{Z}}_{0}\]
is exact. 
Therefore, multiplication by $p$ is an automorphism of $Q_u$ for each prime p,
or, in other words, $Q_u$ has the structure of a $\bQ$-vector space.
This turns $Q_u$ into a $\bQ[x]$-module.

Finally, suppose $Q_u\neq 0$. 
Choose a maximal submodule $\mathfrak{m}\subseteq Q_u$ and consider $Q_u / \mathfrak{m}$.
It is a simple $\bQ [x]$-module, so it is a number field.
Now observe that all  $I_u$, $I_v$, $Q_u$ and $Q_u / \mathfrak{m}$ are finitely generated $\bZ [x]$-modules.
A contradiction, since a number field is not a finitely generated $\bZ [x]$-module.
\end{proof}  

\section{Example: Homogeneous spaces of dimension \texorpdfstring{$75$}{75}}\label{Example}
In light of Proposition~\ref{Xu_dim}, it is interesting to consider spaces dimension by dimension.
Suppose $\dim X_u(G)=75$.
There are no simple groups of dimension 75 but there are three groups of dimension 78: types $B_6$, $C_6$ and $E_6$.
It follows that $u\neq 1$: there are 93 potentially different spaces (the same as the number of non-trivial unipotent classes in these groups).

Theorem~\ref{main_theorem} splits the spaces into three buckets: spaces coming from different groups are not homotopy equivalent.
Then we can make use of Dynkin indices, cf. Proposition \ref{low_pi}.
In the case of $G$ of type $E_6$ all Dynkin indices are distinct, cf. \cite[Table 18]{Dynkin1952}.
Thus, all 20 non-trivial unipotent classes give 
(homotopically) distinct space $X_u$.

In the cases of $B_6$ and $C_6$, the numbers of orbits are 34 and 39 correspondingly.
The Dynkin index does not distinguish all spaces completely, leaving only a small number of cases to consider.
Note that some of them are distinguished using $\pi_2 (X_u)$, cf. Proposition \ref{low_pi}.
\begin{itemize}
    \item Remaining orbits in type $B_6$ and their Dynkin indices
    {
\renewcommand{\arraystretch}{1.2} 

\begin{longtable}{ |P{3.5cm}| P{3.5cm} |P{2cm}|P{2cm}|}
\hline
 Partition of $13$ & Dynkin diagram \newline              \begin{tikzpicture}[scale=1]
                \dynkin B{6}
                \node[above=0.5ex] at (root 1) {$1$};
                \node[above=0.5ex] at (root 2) {$2$};
                \node[above=0.5ex] at (root 3) {$3$};
                \node[above=0.5ex] at (root 4) {$4$};
                \node[above=0.5ex] at (root 5) {$5$};
                \node[above=0.5ex] at (root 6) {$6$};
              \end{tikzpicture}
  & Dynkin Index & Quite even? (Y/N)\\
  \hline
    $[5^2, 1^3]$ & $\triangle(0, 2, 0, 2, 0, 0)$ & $20$ & Y\\
    \hline
    $[5, 4^2]$ & $\triangle(1, 0, 1, 1, 0, 1)$ & $20$ & N\\
    \hline
    $[5, 3, 1^5]$ & $\triangle(2, 0, 2, 0, 0, 0)$ & $12$ & Y\\
    \hline
    $[5, 2^4]$ & $\triangle(2, 1, 0, 0, 0, 1)$ & $12$ & N\\
    \hline    
    $[4^2, 3, 1^2]$ & $\triangle(0, 1, 1, 0, 1, 0)$ & $12$ & N\\
    \hline
    $[5, 2^2, 1^4]$ & $\triangle(2, 1, 0, 1, 0, 0)$ & $11$ & N\\
    \hline
    $[4^2, 2^2, 1]$ & $\triangle(0, 2, 0, 0, 0, 1)$ & $11$ & N\\
    \hline
    $[5, 1^8]$ & $\triangle(2, 2, 0, 0, 0, 0)$ & $10$ & N\\
    \hline
    $[4^2, 1^5]$ & $\triangle(0, 2, 0, 1, 0, 0)$ & $10$ & N\\
    \hline
    $[3^2, 1^7]$ & $\triangle(0, 2, 0, 0, 0, 0)$ & $4$ & Y\\
    \hline
    $[3, 2^4, 1^2]$ & $\triangle(1, 0, 0, 0, 1, 0)$ & $4$ & N\\
    \hline
    $[3, 2^2, 1^6]$ & $\triangle(1, 0, 1, 0, 0, 0)$ & $3$ & N\\
    \hline
    $[2^6, 1]$ & $\triangle(0, 0, 0, 0, 0, 1)$ & $3$ & N\\
    \hline
    $[3, 1^{10}]$ & $\triangle(2, 0, 0, 0, 0, 0)$ & $2$  & N\\
    \hline
    $[2^4, 1^5]$ & $\triangle(0, 0, 0, 1, 0, 0)$ & $2$ & N\\
    \hline
\end{longtable}
}
    \item Remaining orbits in type $C_6$ and their Dynkin indices.
    {
    
\renewcommand{\arraystretch}{1.2} 
\begin{longtable}{ |P{3.5cm}| P{3.5cm}|P{2cm}|P{2cm}|}
 \hline
 Partition of $12$ & Dynkin diagram \begin{tikzpicture}[scale=1]
                \dynkin C{6}
                \node[above=0.5ex] at (root 1) {$1$};
                \node[above=0.5ex] at (root 2) {$2$};
                \node[above=0.5ex] at (root 3) {$3$};
                \node[above=0.5ex] at (root 4) {$4$};
                \node[above=0.5ex] at (root 5) {$5$};
                \node[above=0.5ex] at (root 6) {$6$};
              \end{tikzpicture} & Dynkin Index & Quite even? (Y/N) \\
    \hline
    $[4, 2, 1^6]$ & $\triangle(2, 0, 1, 0, 0, 0)$ &$11$ & N \\
    \hline
    $[3^2, 2^3]$ & $\triangle(0, 1, 0, 0, 1, 0)$ &$11$ & N\\
    \hline
    $[4, 1^8]$ & $\triangle(2, 1, 0, 0, 0, 0)$ &$10$ & N\\
    \hline
    $[3^2, 2^2, 1^2]$ & $\triangle(0, 1, 0, 1, 0, 0)$ & $10$ & N\\
    \hline
\end{longtable}
}
\end{itemize}

Finally, we can deploy Theorem~\ref{2nd_thm}
to help distinguish the remaining cases.
We pick a prime $p$ and compute (cf. Section~\ref{code} for codes) 
  $\gcd(\overline{y_1}-d_1,\ldots, \overline{y_r}-d_r) \in \bF[x]$, 
the generator of the ideal
$\overline{I_u}$.
  This does not help in type $C_6$ and in one of the cases in type  $B_6$.

\renewcommand{\arraystretch}{1.2} 

\begin{longtable}{ |P{2.5cm}| P{2.5cm} |P{1.5cm}| P{4.5cm}|}
\hline
 Partition of $13$ & Dynkin index
  & char($\mathbb{F}$) & $\overline{I_u}$\\
  \hline
    $[5, 2^4]$ & $12$ & $3$ & $\left((x-2)^3\right)$ \\
    \hline
    $[4^2 , 3, 1^2]$ & $12$ & $3$ & $\left((x-2)^2\right)$\\
    \hline
    $[5, 2^2, 1^4]$ & $11$ & $-$ & no suitable $p$ found \\
    \hline
    $[4^2, 2^2, 1]$ & $11$ & $-$ & no suitable $p$ found \\
    \hline
    $[5, 1^8]$ & $10$ & $2$ & $\left(x^2\right)$\\
    \hline
    $[4^2, 1^5]$ & $10$ & $2$ & $\left(x^4\right)$\\
    \hline
    $[3, 2^2, 1^6]$ & $3$ & $3$ & $\left((x-2)^2\right)$ \\
    \hline
    $[2^6, 1]$ & $3$ & $3$ & $\left((x-2)^3\right)$ \\
    \hline
    $[3, 1^{10}]$ & $2$  & $2$ & $\left(x^2\right)$\\
    \hline
    $[2^4, 1^5]$ & $2$ & $2$ & $\left(x^4\right)$\\
    \hline
\end{longtable}

Therefore, in the case of homogeneous spaces of dimension $75$ we have

\begin{prop}\label{finalprop}
  Let $G, K$ be simply-connected groups of types $B_6$, $C_6$ or $E_6$. Suppose $u \in G$, $v \in K$ are unipotent elements from different classes.
  Then $X_u(G) \not\simeq X_{v}(K)$ in all but the three cases, where it is undetermined:
\begin{enumerate}
    \item $G=K=\Spi_{13}(\bC)$, $u,v$ correspond to partitions $[5,2^2,1^4]$ and $[4^2,2^2,1]$,
    \item $G=K=\Sp_{12} (\bC)$, $u,v$ correspond to partitions $[4,2,1^6]$ and $[3^2,2^3]$,
    \item $G=K=\Sp_{12} (\bC)$, $u,v$ correspond to partitions $[4,1^8]$ and $[3^2,2^2,1^2]$.
\end{enumerate}
\end{prop}

\begin{question}
What techniques can be used to determine whether or not the three undetermined pairs of spaces in Proposition~\ref{finalprop} are homotopy equivalent?   
\end{question}

\section{Example: The effect of Dynkin diagram automorphisms}\label{Example2}
Now we describe the effect of diagram automorphisms on Conjecture~\ref{wk_rigid}:
cf. \cite{collingwood1993nilpotent} for all the relevant information.

In types $E_6$ and $A_r$, $r\geq 2$ the automorphism group of the diagram is $S_2$, of order 2. We remark that we write $S_2$, rather than the more common $C_2$, to avoid confusion between the cyclic group of order $2$, and the group of type $C$, rank $2$.
However, there is no effect on Conjecture~\ref{wk_rigid}: distinct conjugacy classes are not conjugate under outer automorphisms.

In type $D_r$, $r\geq 5$ the automorphism group of the diagram is also $S_2$ but there exist
distinct conjugacy classes conjugate under outer automorphisms.
Each {\em very even partition} (those only with even parts) correspond to the two distinct classes,
conjugate under an outer automorphism.
Thus, the  Conjecture~\ref{wk_rigid} proposes a bijection between homotopy types of spaces $X_u$ and partitions in
$\mP(2r)$, where every even part $d_i$ appears even number of times.

In type $D_4$, we have triality: the automorphism group of the diagram is $S_3$.
There are two very even partitions $[2^4]$ and $[4^2]$, each giving a pair of classes in $\Spi_8 (\bC)$.
However, each pair is conjugate to a third class under the triality:
the classes with partitions $[3,1^5]$ and $[2^4]$ are conjugate,
ditto for $[5,1^3]$ and $[4^2]$.
Thus, the  Conjecture~\ref{wk_rigid} proposes a bijection between homotopy types of spaces $X_u$ and
partitions in $\mP(8)$, which have odd parts and all even parts appear even number of times.

In fact, the conjecture holds in this case. There are five non-trivial partitions left to consider:
$[2^2,1^4]$, $[3,2^2,1]$, $[3^2,1^2]$, $[5,3]$ and $[7,1]$. Their Dynkin indices are all distinct (computed by
\cite[Th. 2.1]{PANYUSHEV201515}): 1, 3, 4, 12 and 28.

Let us examine the aforementioned trios of unipotent elements $u_1,u_2,u_3 \in \Spi_8 (\bC)$, related by triality,
in light of Corollary~\ref{remaining_groups}.
By Theorem~\ref{QuiteEvenClassical}, all $u_i$ are even, not quite even elements.
Thus, in $\Spi_8 (\bC)$ they correspond to three subgroups
$\SL_2 (\bC)$ that are related via the outer automorphisms. Their $-I_2$ are the three different elements of $Z(G) \cong K_4$.
In any of the three quotients $\Spi_8 (\bC)/S_2$, one of the subgroups turns into $\PSL_2 (\bC)$, while the other two stay as $\SL_2 (\bC)$.

\section{Computer Code} \label{code}
The first author has put all the codes used into his GitHub repository \cite{GitHub_DJ}.
For Koszul Complex, it utilises {\em Sagemath} that contains a {\em KoszulComplex} class which can be used alongside the {\em ascii{\_}art} function to display these differentials nicely.
For Lie algebra calculations, the code is is written in {\em Sage} and makes use of the {\em LiE} package.

\section{Declarations}

{\bfseries Conflict of Interest:} The authors declare no competing interests.

{\bfseries Data Availability:} Data sharing not applicable to this article as no datasets were generated or analysed during the current study.

\bibliography{refs_TOP}
\bibliographystyle{plain}

\end{document}